\documentclass[reqno,english]{amsart}
\usepackage[utf8]{inputenc}
\usepackage[normalem]{ulem}
\usepackage{amsmath}
\usepackage{amsthm}
\usepackage{amsfonts}
\usepackage{amssymb}
\usepackage{indentfirst}
\usepackage{graphicx}
\usepackage{float}
\numberwithin{equation}{section}
\newtheorem{theorem}{Theorem}[section]
\newtheorem{lemma}[theorem]{Lemma}

\newtheorem{remark}{Remark}[section]

\usepackage{mathtools}
\usepackage{thmtools}
\declaretheoremstyle[headfont=\normalfont]{normalhead}
\usepackage{color}
\usepackage{morefloats}
\usepackage{hyperref}
\usepackage{textcomp}
\usepackage{url}

\title[Sharp Beckner's Inequality]{On Sharp Beckner's Inequality for Axially Symmetric Functions on $\mathbb{S}^4$}

\author{Tuoxin Li}
\address{Department of Mathematics, University of British Columbia,Vancouver, Canada}
\email{tuoxin@math.ubc.ca}

\author{Juncheng Wei}
\address{Department of Mathematics, University of British Columbia,Vancouver, Canada}
\email{jcwei@math.ubc.ca}

\author{Zikai Ye}
\address{Department of Mathematics, University of British Columbia,Vancouver, Canada}
\email{yezikai@math.ubc.ca}
\begin{document}

\maketitle

\begin{abstract}
We prove that axially symmetric solutions to the $Q$-curvature type problem
$$ \alpha P_4 u + 6(1-\frac{e^{4u}}{\int_{\mathbb{S}^4} e^{4u}})=0 \ \ \ \ \ \mbox{on} \ \mathbb{S}^4 $$
must be constants, provided that $ \frac{1}{2}\le \alpha <1$. This result is sharp in view of the existence of nonconstant solutions to the equation by  Gui,Hu and Xie \cite{GHX2021} for $\frac{1}{5}< \alpha <\frac{1}{2}$.  As a consequence, we prove a sharp  Beckner’s inequality on
$\mathbb{S}^4$ for axially symmetric functions with  center of mass at the origin. This answers an open question in \cite{GHX2021} in which the corresponding results were proved for $ \alpha \geq 0.517$. To close the gap, we make use of some quantitative properties of Gegenbauer polynomials. One of the key ingredients in our proof is the pointwise estimate of large parameters asymptotic expansions for Gegenbauer polynomials proved in Nemes and Daalhuis \cite{Nemes2020}.
\end{abstract}

\section{Introduction and Main Results}
On $\mathbb{S}^4$, the Beckner's inequality (\cite{Beckner1993}), a higher order Moser-Trudinger inequality, says that the functional
\begin{equation*}
J_{\alpha}(u):=\frac{\alpha}{2}\left(\int_{\mathbb{S}^4}|\Delta u|^2
\text{d}w+2\int_{\mathbb{S}^4}|\nabla u|^2
\text{d}w\right)+6\int_{\mathbb{S}^4}u
\text{d}w-\frac{3}{2}\ln\int_{\mathbb{S}^4}e^{4u}
\text{d}w
\end{equation*}
is non-negative, for $\alpha =1$ and all $u\in H^2(\mathbb{S}^4)$, where d$w$ is the normalized Lebesgue measure on $\mathbb{S}^4$ with $\int_{\mathbb{S}^4}\text{d}w=1$.  On the other hand, an improved higher order Moser-Trudinger-Onofri type inequality holds if the center of the mass of $u$ is at the origin:  for  $u$ belonging to the set
\begin{equation*}
\mathcal{L}=\left\{u\in H^2(\mathbb{S}^4)\ :\ \int_{\mathbb{S}^4}e^{4u} x_j \text{d}w=0,\ j=1,..., 5\right\},
\end{equation*}
 and for any $\alpha\geq \frac{1}{2}$, there exists some constant $C(\alpha)\geq0$, such that $J_{\alpha}(u)\geq -C(\alpha)$. Here $x_j$'s are the coordinate components of $\mathbb{R}^5$. As in second order case (\cite{ChangYang1987}), it is conjectured that $C(\alpha)$ can be chosen to be $0$ for any $\alpha\geq \frac{1}{2}$.

The Euler-Lagrange equation of $J_\alpha$ is the constant $Q$-curvature-type equation on $\mathbb{S}^4$
\begin{equation}\label{paneitz}
\alpha P_4 u+6(1-\frac{e^{4u}}{\int_{\mathbb{S}^4}e^{4u}\text{d}w})=0 \ \mbox{on} \ \mathbb{S}^4,
\end{equation}
where $P_4=\Delta^2-2\Delta$ is the Paneitz operator on $\mathbb{S}^4$. The conjecture holds if the equation (\ref{paneitz}) admits only constant solutions. For $\alpha<1$ and close to $1$, it is proved by the second author and Xu \cite{WX1998} that all solutions to (\ref{paneitz}) are constants. It remains open for general $ \alpha \in [\frac{1}{2}, 1)$. For results and backgrounds on $Q$-curvature problems, we refer to \cite{ChangYang1995, ChangYang1997, DHL2000,DM2008,GHX2021, GurMal2015, LiXiong2019,  Mal2006, WX1998} and the references therein.

The counterpart of this problem on $\mathbb{S}^2$ is the so-called Nirenberg problem
$$ -\alpha \Delta u + 1-  \frac{e^{2u}}{\int_{\mathbb{S}^2} e^{2u}}=0 \ \ \mbox{on} \ \mathbb{S}^2$$
and it has received lots of attention in the last four decades. We refer to \cite{ChangYang1987, ChangYang1988, JinLiXiong2017} and the references therein.
It is conjectured by A.Chang and P.Yang \cite{ChangYang1987, ChangYang1988} that  the following functional
\begin{equation*}
\alpha\int_{\mathbb{S}^2}|\nabla u |^2\text{d}w+2\int_{\mathbb{S}^2} u\text{d}w- \ln \int_{\mathbb{S}^2}e^{2u} \text{d}w
\end{equation*}
is non-negative for any $\alpha\geq\frac{1}{2}$ and $u$ with zero center of mass $\int_{\mathbb{S}^2}e^{2u}\Vec{x}\text{d}w=0$. Feldman, Froese,Ghoussoub and Gui \cite{FFGG1998} proved that the conjecture is true for axially symmetric functions when $\alpha >\frac{16}{25}-\epsilon$. Gui and the second author  \cite{GW2000} proved that the conjecture is true for axially symmetric functions. Ghoussoub and Lin \cite{GL2010} showed that the conjecture is true for $\alpha >\frac{2}{3}-\epsilon$. Finally,   Gui and Moradifam \cite{GM2018} proved that the conjecture is indeed true.  See \cite{ChangHang2022, ChangGui202} for references and more general results on improved Moser-Trudinger-Onofri inequality  on $\mathbb{S}^2$ and relations with Szeg\"o limit theorem.

In this paper, we will study axially symmetric solution $u$ to \eqref{paneitz} for $ \alpha \in [\frac{1}{2}, 1)$. As in \cite{GHX2021},  \eqref{paneitz} becomes
\begin{equation}\label{axial}
\alpha[(1-x^2)^2u']'''+6-\frac{8e^{4u}}{\int_{-1}^{1}(1-x^2)e^{4u}}=0,\ x\in(-1,1),
\end{equation}
which is the critical point of the functional
\begin{align*}
I_\alpha(u)
&=\frac{\alpha}{2}\int_{-1}^{1}((1-x^2)|(1-x^2)u''|^2+6|(1-x^2)u'|^2)\\
&+6\int_{-1}^{1}(1-x^2)u-2\ln(\frac{3}{4}\int_{-1}^{1}(1-x^2)e^{4u})
\end{align*}
restricted to the set
\begin{equation}
\mathcal{L}_r=\{u\in H^2(\mathbb{S}^4):\ u=u(x)\text{ and }\int_{-1}^{1}x(1-x^2)e^{4u} dx=0\}.
\end{equation}

Our main result is
\begin{theorem}\label{main}
If $\alpha\geq \frac{1}{2}$, then the only critical point of the functional $I_\alpha$ restricted to $\mathcal{L}_r$ are constant functions. As a consequence we have the following improved Beckner's inequality for axially symmetric functions on $\mathbb{S}^4$
$$ \inf_{ u\in {\mathcal{L}_r}} I_\alpha (u)=0, \ \alpha \geq \frac{1}{2}. $$
\end{theorem}

The assumption on $\alpha$ in Theorem \ref{main} is sharp. By bifurcation methods, Gui,Hu and Xie \cite{GHX2021} proved that Theorem \ref{main} fails for $\frac{1}{5}< \alpha <\frac{1}{2}$. They also showed that Theorem \ref{main} holds for $\alpha\geq 0.517$, using similar strategies as in \cite{FFGG1998, GW2000}. More precisely, they expanded $G=(1-x^2)u'$ in terms of Gegenbauer polynomials and introduced a quantity $D$ (see \eqref{402}) and obtained an inequality for each $n\geq 3$. From these inequalities, they expected to go through similar induction procedure as in  \cite{GW2000}. However, the estimates of Gegenbauer coefficients of $G$ they have obtained are not refined enough to make the induction procedure work for large modes. As a consequence, one cannot obtain the optimal constant $\alpha \geq \frac{1}{2}$ by their method. See the discussion in [Section 6, \cite{GHX2021}] for more details.

We will use the same strategy, as in \cite{GW2000, GHX2021}, but with more refined estimates on Gegenbauer coefficients of $G$. In particular we make use of pointwise estimates proved in [Corollary 5.3, Nemes and Daalhuis\cite{Nemes2020}].  By refining the behavior of Gegenbauer polynomial near $x=\pm 1$ and using the decaying properties away from $x=\pm 1$, we  show that the induction procedure in \cite{GW2000} still works in this setting.

Under similar settings, this problem can be generalized to $\mathbb{S}^n, n\geq 3$. Gui, Hu and Xie \cite{GHW2022} showed that the counterparts of Theorem \ref{main} fails for $\frac{1}{n+1}\leq \alpha<\frac{1}{2}$. When $n=6,\ 8$, they showed that for $\alpha \geq  0.6168$ ($n=6$) and $ \alpha \geq  0.8261$ ($n=8$), all critical points are constants.  Whether or not the optimal constant is $\alpha=\frac{1}{2}$ remains unknown. We believe that our estimates in this paper can give a unified proof for sharp Beckner's inequality on $\mathbb{S}^n$, at least in the axially symmetric case. We will return to this in a future work.

The organization of this paper is as follows. In Section 2 we collect some properties of Gegenbauer polynomials and the expansions of $G=(1-x^2)u'$ (proved in \cite{GHX2021}). In Section 3 we give  refined estimates on the Gegenbauer coefficients of $G$ (Theorem \ref{A}). In Section 4 we prove the main Theorem \ref{main} by induction argument. Three technical lemmas (Lemmas \ref{lem32}, \ref{lem34} and \ref{lem35}) are proved in the appendices.

\section{Preliminaries and some basic estimates}
In this section, we collect some properties of Gegenbauer polynomials and some known facts on equation (\ref{axial}).

Recall that the Gegenbauer polynomials of  order $\nu$ and degree $k$ (\cite{Mori1998}) is given by
\begin{equation*}
C_{k}^{\nu}(x)=\frac{(-1)^k}{2^k n!}\frac{\Gamma(\nu+\frac{1}{2})\Gamma(k+2\nu)}{\Gamma(2\nu)\Gamma(\nu+k+\frac{1}{2})}(1-x^2)^{-\nu+\frac{1}{2}}\frac{d^k}{dx^k} (1-x^2)^{k+\nu-\frac{1}{2}}.
\end{equation*}

The derivative of $C_{k}^{\nu}$ satisfies
\begin{equation}\label{201}
\frac{d}{dx}C_{k}^{\nu}(x)=2\nu C_{k-1}^{\nu+1}(x).
\end{equation}

On $\mathbb{S}^4$ the corresponding Gegenbauer polynomial for bi-Laplacian  is $C_k^{\frac{3}{2}}$. (On $\mathbb{S}^2$ it is $C_k^{\frac{1}{2}}$.)  Let $F_k$ be the normalization of $C_{k}^{\frac{3}{2}}$ such that $F_k(1)=1$. More precisely,
\begin{equation*}
F_k:=\frac{2}{(k+1)(k+2)}C_{k}^{\frac{3}2}.
\end{equation*}
The first few terms of $F_k$ are given as follows
\begin{equation*}
F_0=1,\ F_1=x,\ F_2=\frac{5}{4}x^2-1,\ F_3=\frac{7}{4}x^3-\frac{3}{4}x.
\end{equation*}

Also, $F_k$ satisfies
\begin{equation}
(1-x^2)F_k''-4xF_k'+\lambda_k F_k=0, \lambda_k=k(k+3)
\end{equation}
and
\begin{equation}
\int_{-1}^{1}(1-x^2)F_{k}F_l=\frac{8}{(2k+3)(k+1)(k+2)}\delta_{kl}.
\end{equation}

As in \cite{GHX2021, GW2000}, we define the following key quantity
\begin{equation}
G(x)=(1-x^2)u',
\end{equation}
where $u$ is a solution to \eqref{axial}. Then $G$ satisfies the equation
\begin{equation}
\alpha((1-x^2)G)'''+6-\frac{8}{\gamma}e^{4u}=0,
\end{equation}
where
\begin{equation}
    \gamma=\int_{-1}^{1}(1-x^2)e^{4u}.
\end{equation}

We can expand $G$ in terms of the Gegenbauer polynomials $F_k$:
\begin{equation}
\label{Gexpand}
G=a_0F_0+\beta x+a_2F_2(x)+\sum_{k=3}^{\infty}a_kF_k(x).
\end{equation}

Denote
\begin{equation}
\label{gdef}
g= (1-x^2) \frac{e^{4u}}{\gamma}, \ a:=\int_{-1}^1 (1-x^2)g.
\end{equation}

We recall the following results from  \cite{GHX2021}:
\begin{lemma}[Lemma 2.2 in \cite{GHX2021}] For $g=(1-x^2)\frac{e^{4u}}{\gamma}$ and $G=(1-x^2)u'$ as above, we have $a_0=0$ and
\begin{equation}\label{}
\int_{-1}^{1}(1-x^2)F_1G=\frac{4}{15}\beta,
\end{equation}
\begin{equation}
a=\int_{-1}^{1}(1-x^2)g=\frac{4}{5}(1-\alpha\beta),
\end{equation}
\begin{equation}\label{bk}
\int_{-1}^{1}(1-x^2)F_k G=-\frac{8}{\alpha\lambda_k(\lambda_k+2)}\int_{-1}^{1}(1-x^2)gF_{k}',\text{ }k\geq 2,
\end{equation}
\begin{equation}\label{G'}
\int_{-1}^{1}|[(1-x^2)G]'|^2=\frac{16}{15}(5-\frac{1}{\alpha})\beta.
\end{equation}
\end{lemma}

\begin{lemma}[Lemma 3.1 and Lemma 3.2 in \cite{GHX2021}]\label{floor}
There holds
\begin{equation}
\label{Gbd}
G' \leq \frac{1}{\alpha},
\end{equation}
\begin{equation}
\lfloor G\rfloor^2\leq (\frac{4}{\alpha}-6)\int_{-1}^{1}|[(1-x^2)G]'|^2+\frac{16}{\alpha}\int_{-1}^{1}(1-x^2)G^2,
\end{equation}
where
\begin{equation}
\lfloor G\rfloor^2= \int_{-1}^{1}(1-x^2)[(1-x^2)^2G']'''G.
\end{equation}
\end{lemma}

\section{Refined Estimates on $b_k$'s}

Let $b_{k}=a_{k} \sqrt{ \int_{-1}^{1}(1-x^2)F_{k}^{2}}$, where $a_k$ is the $k$-th coefficient in the expansion of $G$ (see (\ref{Gexpand})). The estimates of $b_k$ play a key role in the proofs of \cite{GHX2021, GW2000}. In \cite{GHX2021}, they used \eqref{bk} and the fact that
\begin{equation}
|F_k'(x)|\leq |F_k'(1)|=\frac{\lambda_k}{4}
\end{equation}
to estimate $b_k$ as follows
\begin{align*}
b_{k}^{2}&=a_{k}^2\int_{-1}^{1}(1-x^2)F_{k}^{2}=\frac{1}{\int_{-1}^{1}(1-x^2)F_{k}^{2}}\left[\frac{8}{\alpha\lambda_k(\lambda_k+2)}\int_{-1}^{1}(1-x^2)gF_{k}'\right]^2\\
&\leq \frac{(2k+3)(k+1)(k+2)}{8}\left[\frac{8}{\alpha\lambda_k(\lambda_k+2)}\frac{\lambda_k}{4}a\right]^2\\
&=\frac{2k+3}{2\alpha^2(\lambda_k+2)}a^2.
\end{align*}

As discussed in Section 6 of \cite{GHX2021}, the above estimates are not sufficient to deduce the induction
\begin{equation}
\label{a1}
a=\frac{4}{5} (1-\alpha \beta) \leq \frac{d}{\lambda_{k}}
\end{equation}
as in \cite{GW2000}. With the bounds for $b_k$ the induction (\ref{a1}) fails for large $k$.

The above discussions motivate us to find more refined estimate on $|F_k'|$, and hence on $b_k$, for large $k$. A key observation is that  $F_{k}'$ attains its maximum only at $\pm 1$ and it decays rapidly away from $\pm 1$. See Figure \ref{figure1} and Figure \ref{figure2} below. As a result, we can improve $a$ to be $a-0.089\lambda_ka^2$. For simplicity, in the rest of the paper, we will denote
\begin{equation}
\widetilde{F}_k':=\frac{4}{\lambda_k}F_k'=\frac{24}{\lambda_k(\lambda_k+2)}C_{k-1}^{\frac{5}{2}}
\end{equation}
so that $\widetilde{F}_k'(1)=1$. One way to improve the estimate in (\ref{bk}) is to split the right hand integral in (\ref{bk}) into two parts. To this end, we define
\begin{equation}
a_+:=\int_{0}^{1}(1-x^2)g, \ a_-:=\int_{-1}^{0}(1-x^2)g,
\end{equation}
where we recall  $g=(1-x^2)\frac{e^{4u}}{\gamma}, \ a=\int_{-1}^1 (1-x^2)g=a_{+}+a_{-}$.  Without loss of generality, we may assume $a_+=\lambda a$ with $\frac{1}{2}\le \lambda \le 1$.

The following theorem gives the key refined estimate for $b_k$.
\begin{theorem}\label{A}
Let $A_k:=\int_{-1}^{1}(1-x^2)g \widetilde{F}_{k}'$. If $a\le \frac{5}{\lambda_n}$, then for all $2\le k\le n$,
\begin{equation}\label{301}
|A_k|\leq
\begin{cases}
(0.081+0.919\lambda)a-0.089\lambda^2\lambda_k a^2&\text{ if }k \text{ is even},\\
a-0.089\lambda_k(2\lambda^2-2\lambda+1)a^2&\text{ if }k \text{ is odd}.
\end{cases}
\end{equation}
As a consequence, $b_k$ satisfies
\begin{equation}\label{302}
b_{k}^{2}\leq \frac{2k+3}{2\alpha^2(\lambda_k+2)}
\begin{cases}
{[}(0.081+0.919\lambda)a-0.089\lambda^2\lambda_k a^2]^2&\text{ if }k \text{ is even},\\
{[}a-0.089\lambda_k(2\lambda^2-2\lambda+1)a^2]^2&\text{ if }k \text{ is odd}.
\end{cases}
\end{equation}
\end{theorem}
\begin{remark}
The reason that we have to expand the estimate of $b_k$ to the next order term is that for large $k$, the induction region is $ a \lambda_k \sim O(1)$ and hence the next order term should not be neglected.
\end{remark}

Before we prove Theorem \ref{A}, we can first consider some cases where $k$ is small. In fact, using the fact that $\int_{-1}^1 x (1-x^2) e^{4u}=0$  we can obtain much better estimates for small $k$'s. The proof is left to Appendix \ref{appendix}.

\begin{lemma}\label{lem32}
Let $A_k$ be as in Theorem \ref{A}. Then
\begin{eqnarray}
    |A_2|&\le& a_+\sqrt{1-\frac{2a_+}{a+1}},\label{303}\\
    |A_3|&\le& \max\{ a-\frac{7}{3}\frac{a^2}{a+1}(2\lambda^2-2\lambda+1), \frac{a}{6}\},\label{304}\\
    |A_4|&\le&  A_{1,1}^+-3(A_{1,1}^+)^2+\frac{1}{9}a_-,\label{305}\\
    |A_5|&\le& a-\frac{6(a_+^2+a_-^2)}{a+1}+\frac{33(a_+^3+a_-^3)}{4(a+1)^2}.\label{306}
\end{eqnarray}
\end{lemma}
\begin{remark}\label{rmk32}
    It is easy to see that the estimate of $|A_2|$ \eqref{303} becomes the worst when $\lambda=1$; while the estimate of $|A_3|$ \eqref{304} becomes the worst when $\lambda=\frac{1}{2}$. The same is true for $|A_4|$ and $|A_5|$ provided that $a$ is suitably small. More generally, in \eqref{301} and \eqref{302}, one can easily show that the estimates become the worst when $\lambda=1$ if $k$ is even, and when $\lambda=\frac{1}{2}$ if $k$ is odd. In fact, we can say more about this. See Lemma \ref{Lemma4.1} below.
\end{remark}

Now we derive some estimates about $g$. By definition, $g=(1-x^2)\frac{e^{4u}}{\gamma}$,
\begin{equation*}
\int_{-1}^{1}g=1,\ \int_{-1}^{1}xg=0\text{ and }\int_{-1}^{1}(1-x^2)g=a.
\end{equation*}

From the second integral in the above, we have
\begin{equation}\label{307}
\int_{0}^{1}g-\int_{0}^{1}(1-x)g=\int_{0}^{1}xg=-\int_{-1}^{0}xg=\int_{-1}^{0}g-\int_{-1}^{0}(1+x)g.
\end{equation}

Since
\begin{equation*}
\left|\int_{0}^{1}(1-x)g\right|\leq\int_{0}^{1}(1-x^2)g=a_+,\ \left|\int_{-1}^{0}(1+x)g\right|\leq\int_{0}^{1}(1-x^2)g=a_-,
\end{equation*}
we have
\begin{equation*}
    \left|\int_{0}^{1}g-\int_{-1}^{0}g \right|\le a,
\end{equation*}
hence
\begin{equation}\label{308}
\frac{1-a}{2}\leq\int_{0}^{1}g,\int_{-1}^{0}g\leq \frac{1+a}{2}.
\end{equation}
Moreover,
\begin{equation}\label{309}
    \int_{0}^{1}xg\le \min\{\int_{0}^{1}g,\int_{-1}^{0}g\}\le\frac{1}{2},
\end{equation}
and
\begin{equation}\label{N310}
    \int_0^1 (1+x)g=1-\int_{-1}^0 (1+x)g<1.
\end{equation}

To prove Theorem \ref{A}, we need some point-wise estimates on $\Tilde{F}_k'=\frac{24}{\lambda_k(\lambda_k+2)}C_{k-1}^{\frac{5}{2}}$. The following Lemma gives us the asymptotic behavior of Gegenbauer polynomials.

\begin{lemma}[Corollary 5.3 of Nemes and Daalhuis \cite{Nemes2020} ]\label{lem33}
Let $0<\zeta<\pi$ and $N\geq 2$ be an  integer. Then
\begin{equation}\label{310}
C_{k-1}^{\frac{5}{2}}(\cos{\zeta})=\frac{2}{\Gamma(\lambda)(2\sin{\zeta})^\frac{5}{2}}\left(\sum_{n=0}^{N-1}t_n(2)\frac{\Gamma(k+4)}{\Gamma(k+n+\frac{5}{2})}\frac{\cos{(\delta_{k-1,n})}}{\sin^n{\zeta}}+R_N(\zeta,k-1)\right),
\end{equation}
where $\delta_{k,n}=(k+n+\frac{5}{2})\zeta-(\frac{5}{2}-n)\frac{\pi}{2}$, $t_n(\mu)=\frac{(\frac{1}{2}-\mu)_n(\frac{1}{2}+\mu)_n}{(-2)^n n!}$, and $(x)_n=\frac{\Gamma(x+n)}{\Gamma(x)}$ is the Pochhammer symbol. The remainder term $R$ satisfies the estimate
\begin{equation}\label{311}
|R_N(\zeta,k)|\leq |t_N(2)|\frac{\Gamma(k+5)}{\Gamma(k+N+\frac{7}{2})}\frac{1}{\sin^N{\zeta}}\cdot
\begin{cases}
|\sec{\zeta}|&\text{ if }0<\zeta\leq\frac{\pi}{4}\text{ or }\frac{3\pi}{4}\leq\zeta<\pi,\\
2\sin{\zeta}&\text{ if }\frac{\pi}{4}<\zeta<\frac{3\pi}{4}.
\end{cases}
\end{equation}
\end{lemma}

Using the pointwise estimate (\ref{310}),
we can prove the following lower and upper bounds for $\widetilde{F}_k'$.  The proofs are left to  Appendix \ref{appendixB}.

\begin{lemma}\label{lem34}
For all $k\ge 6$, we have
\begin{eqnarray*}
\widetilde{F}_k'\ge -0.081,\quad 0\le x\le 1.
\end{eqnarray*}
\end{lemma}

\begin{lemma}\label{lem35}
Let $d=10$ and $b=0.11$. Then for all $k\ge 6$,
\begin{eqnarray*}
\widetilde{F}_k'\le\left\{\begin{aligned} 
&b,\quad &0\le x\le1-\frac{d}{\lambda_k},\\
&1-\frac{\lambda_k}{d}(1-b)(1-x), \quad &1-\frac{d}{\lambda_k}\le x \le 1.
\end{aligned}
\right.
\end{eqnarray*}
\end{lemma}
 
 The above two lemmas can be illustrated by the following two figures (Figure 1 and Figure 2): $\widetilde{F}_k'$ decays rapidly away from $1$ and its minimum is very small.

\begin{figure}[htbp]
\centering
\begin{minipage}[t]{0.48\textwidth}
\centering
\includegraphics[width=6cm]{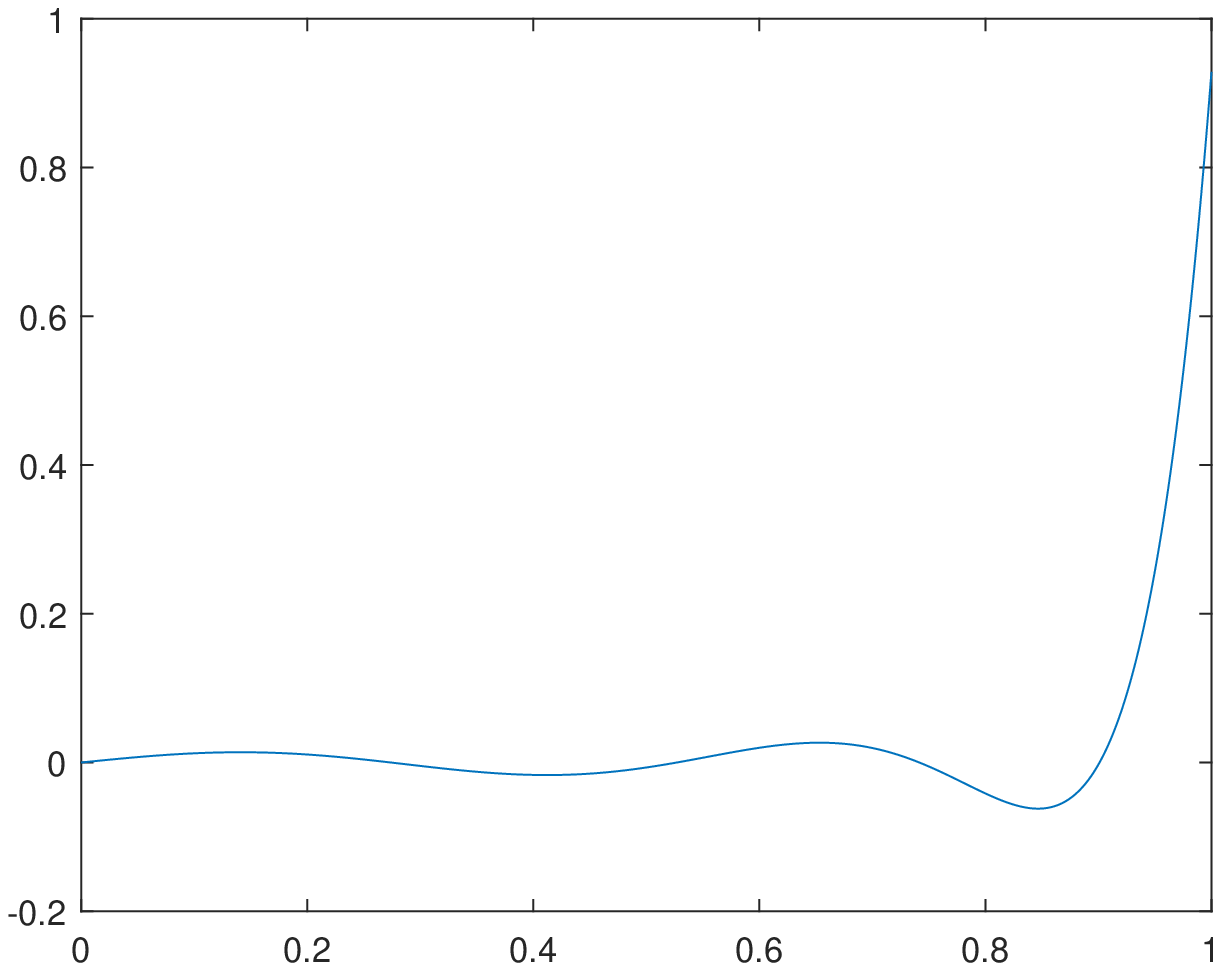}
\caption{Graph of $\widetilde{F}_{10}'$}
\label{figure1}
\end{minipage}
\begin{minipage}[t]{0.48\textwidth}
\centering
\includegraphics[width=6cm]{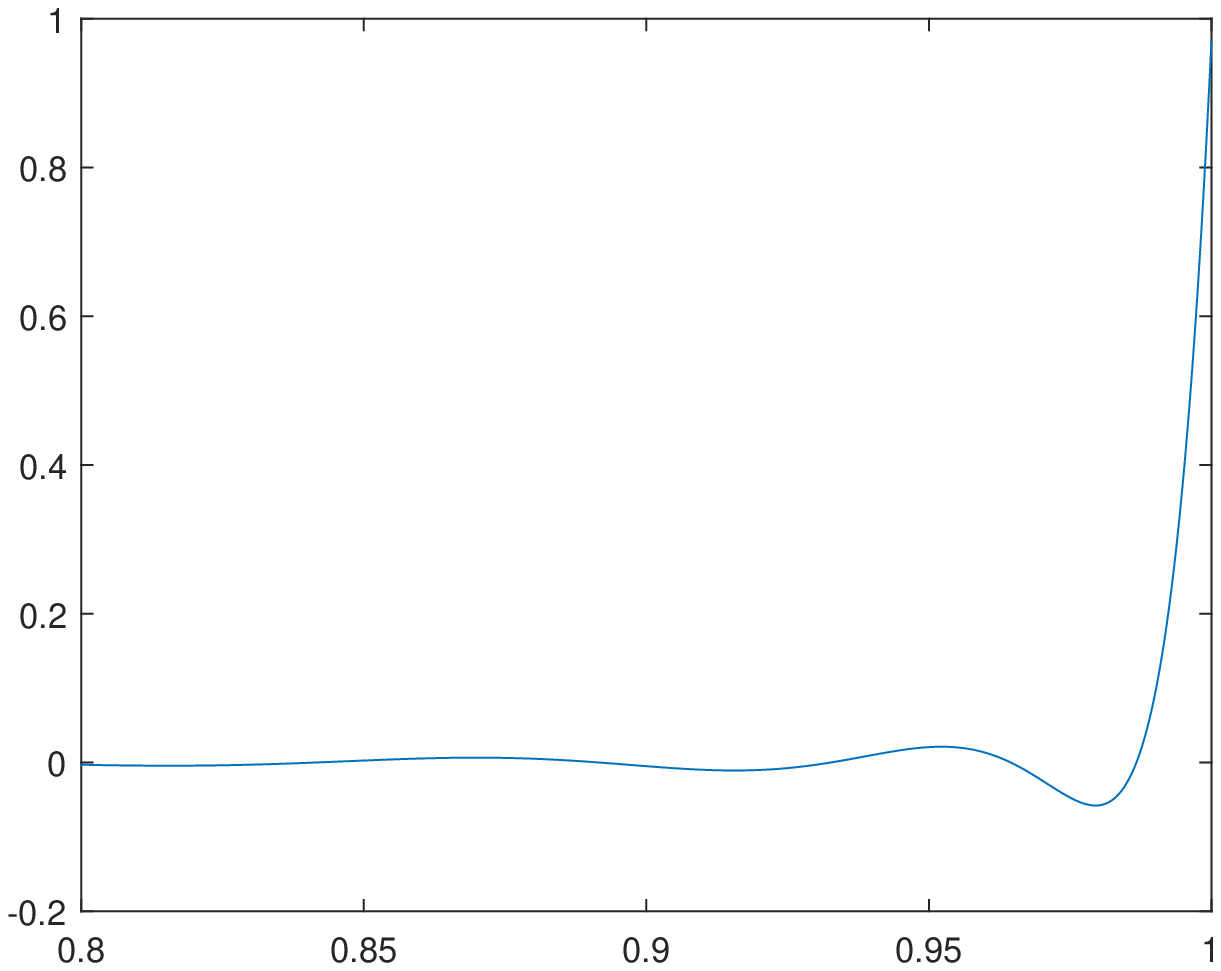}
\caption{$\widetilde{F}_{30}'$ on $(0.8,1)$}
\label{figure2}
\end{minipage}
\end{figure}

With the aid of Lemma \ref{lem34} and Lemma \ref{lem35},  we can prove Theorem \ref{A} and get a refined estimate on $b_k$.

\begin{proof}[Proof of Theorem \ref{A}] By results of \cite{GHX2021}, which is restated in \eqref{403} below, we have  $\beta\geq \frac{16}{13}$ and hence $a<0.31$. It is then straightforward to check that for $2\le k\le 5$, the estimates for $|A_k|$ in Lemma~\ref{lem32} is better than that in Theorem~\ref{A}, so in what follows we may assume $k\ge 6$.

Define $I=(0,1-\frac{d}{\lambda_k})$, $II=(1-\frac{d}{\lambda_k},1)$, and $a_I=\int_I (1-x^2) g$, $a_{II}=\int_{II} (1-x^2) g$. Then there holds
\begin{align*}
    \int_0^1 (1-x^2)\widetilde{F}_k'g&=\int_I (1-x^2)\widetilde{F}_k'g+\int_{II} (1-x^2)\widetilde{F}_k'g\\
    &\le \int_I (1-x^2)b g+\int_{II} (1-x^2)(1-\frac{\lambda_k}{d}(1-b)(1-x))g\\
    &=ba_I+a_{II}-\frac{\lambda_k}{d}(1-b)\int_{II} (1-x^2)(1-x)g\\
    &\le ba_I+a_{II}-\frac{\lambda_k}{d}(1-b)\frac{(\int_{II} (1-x^2)g)^2}{\int_{II} (1+x)g}\\
    &\le ba_I+a_{II}-\frac{\lambda_k}{d}(1-b)a_{II}^2\\
    &=ba_++(1-b)(a_{II}-\frac{\lambda_k}{d}a_{II}^2),
\end{align*}
where we have used \eqref{N310}.

By assumption, $a_{II}\le a_+\le a\le \frac{5}{\lambda_k} =\frac{d}{2\lambda_k}$, so we have
\begin{equation*}
    \int_0^1 (1-x^2)\widetilde{F}_k'g\le a_++(1-b)(a_+-\frac{\lambda_k}{d}a_+^2)=a_+-\frac{\lambda_k}{d}(1-b)a_+^2.
\end{equation*}
On the other hand, by Lemma \ref{lem35},
\begin{equation*}
    \int_0^1 (1-x^2)\widetilde{F}_k'g\ge -0.081\int_0^1 (1-x^2)g=-0.081a_+.
\end{equation*}
Therefore
\begin{equation}\label{312}
    -0.081a_+\le \int_0^1 (1-x^2)\widetilde{F}_k'g\le a_+-\frac{\lambda_k}{d}(1-b)a_+^2.
\end{equation}
Similarly, under the same assumption, if $k$ is odd, then
\begin{equation}\label{313}
    -0.081a_-\le \int_{-1}^0 (1-x^2)\widetilde{F}_k'g\le a_--\frac{\lambda_k}{d}(1-b)a_-^2;
\end{equation}
while if $k$ is even, then
\begin{equation}\label{314}
    -(a_--\frac{\lambda_k}{d}(1-b)a_-^2)\le \int_{-1}^0 (1-x^2)\widetilde{F}_k'g\le 0.081a_-.
\end{equation}
Theorem \ref{A} then follows from \eqref{312}-\eqref{314}.
\end{proof}

\section{proof of Theorem \ref{main}}
In this section,  we will prove Theorem \ref{main} by induction argument, thanks to the better estimates in Theorem \ref{A}.

We claim that $\beta=0$, which yields that $(1-x^2)G$ is constant by \eqref{G'}. Since $G$ is bounded on $(-1,1)$, we get $G\equiv 0$ and we are done.

So it suffices to show that $\beta=0$. We will argue by contradiction. If $\beta\neq 0$, then $0<\beta<\frac{1}{\alpha}$ since $a=\int_{-1}^1 (1-x^2) g=\frac{4}{5}(1-\alpha\beta)>0$. It then suffices to show $a=0$. We will achieve this by proving
 \begin{equation}
 \label{ainduction}
 a=\frac{4}{5}(1-\alpha\beta)\leq \frac{5}{\lambda_n},\ \forall n\geq 3.
 \end{equation}

 As in \cite{GW2000}, we will prove (\ref{ainduction}) by induction.

To begin with, following \cite{GHX2021}, we introduce the following quantity
\begin{equation}
D:=\sum_{k=3}^{\infty}\left[\lambda_k(\lambda_k+2)-(10-\frac{4}{3\alpha})(\lambda_k+2)-\frac{16}{\alpha}\right]b_{k}^2.
\end{equation}

Let $G_2:=\sum\limits_{k=3}^{\infty}a_kF_k(x)$. Then by \eqref{G'} and Lemma \ref{floor}, we get
\begin{align}
D&=\int_{-1}^{1}(1-x^2)[(1-x^2)^2G_2']'''G_2-(10-\frac{4}{3\alpha})\int_{-1}^{1}|((1-x^2)G_2)'|^2-\frac{16}{\alpha}\int_{-1}^{1}G_{2}^{2}\notag\\
&\le \lfloor G\rfloor^2-(10-\frac{4}{3\alpha})\int_{-1}^{1}|((1-x^2)G)'|^2-\frac{16}{\alpha}\int_{-1}^{1}G^2+(36+\frac{8}{\alpha})\beta^2\int_{-1}^{1}F_{1}^{2}\notag\\
&\leq (\frac{16}{3\alpha}-16)\int_{-1}^{1}|((1-x^2)G)'|^2+(36+\frac{8}{\alpha})\frac{4\beta^2}{15}\notag\\
&\leq \frac{16\beta}{15}\left[(9+\frac{2}{\alpha})\beta+(\frac{16}{3\alpha}-16)(5-\frac{1}{\alpha})\right].\label{402}
\end{align}
Since $D\ge 0$, $\alpha\ge \frac{1}{2}$ and $0<\beta<\frac{1}{\alpha}$, we obtain
\begin{equation}\label{403}
    \beta\ge \frac{16}{13}(1-\frac{1}{3\alpha})(5-\frac{1}{\alpha})\ge \frac{16}{13}.
\end{equation}

On the other hand, fix any integer $n\geq 3$, we have
\begin{align}
D&=\sum_{k=3}^{\infty}\left[\lambda_k(\lambda_k+2)-(10-\frac{4}{3\alpha})(\lambda_k+2)-\frac{16}{\alpha}\right]b_{k}^2\notag\\
&\geq\sum_{k=3}^{n}\left[\lambda_k(\lambda_k+2)-(10-\frac{4}{3\alpha})(\lambda_k+2)-\frac{16}{\alpha}\right]b_{k}^2\notag\\
&+(\lambda_{n+1}-10+\frac{4}{5\alpha})\sum_{k=n+1}^{\infty}(\lambda_k+2)b_{k}^{2}\notag\\
&\geq\sum_{k=3}^{n}(\lambda_k-\lambda_{n+1}-\frac{4}{15\alpha})(\lambda_k+2)b_{k}^2+(\lambda_{n+1}-10+\frac{4}{5\alpha})\sum_{k=3}^{\infty}(\lambda_k+2)b_{k}^{2}\notag\\
&=\sum_{k=3}^{n}(\lambda_k-\lambda_{n+1}-\frac{4}{15\alpha})(\lambda_k+2)b_{k}^2+(\lambda_{n+1}-10+\frac{4}{5\alpha})\left[\frac{16\beta}{15}(5-\frac{1}{\alpha})-\frac{8\beta^2}{5}-\frac{8a_{2}^{2}}{7}\right]\label{404}
\end{align}

Combining \eqref{402} and \eqref{404}, we get
\begin{align}\label{405}
0
&\leq \frac{16\beta}{15}(5-\frac{1}{\alpha})(\frac{68}{15\alpha}-6-
\lambda_{n+1})
+\frac{8}{15}\beta^2(3\lambda_{n+1} -12+\frac{32}{5\alpha})\notag\\
&+\frac{8}{15\alpha}(\lambda_2+2)b_2^2+\sum_{k=2}^{n}(\lambda_{n+1}-\lambda_k+\frac{4}{15\alpha})(\lambda_k+2)b_{k}^{2}.
\end{align}

Now we can start the induction procedure. First we rewrite \eqref{405} in terms of $a$ and $\alpha$:
\begin{align}\label{406}
    a(4-5a)(3\lambda_{n+1} -12+\frac{32}{5\alpha})&\le \frac{4(4-5a)}{15\alpha}(-8+136\alpha-180\alpha^2-15\lambda_{n+1}(2\alpha^2-\alpha))\notag\\
    &+6\alpha^2\big(\frac{8}{15\alpha}(\lambda_2+2)b_2^2+\sum_{k=2}^{n}(\lambda_{n+1}-\lambda_k+\frac{4}{15\alpha})(\lambda_k+2)b_{k}^{2}\big).\notag\\
\end{align}

By \eqref{403}, we have $a=\frac{4}{5}(1-\alpha\beta)<0.31$. When $n=3$, we apply Lemma ~\ref{lem32} to \eqref{406}, as mentioned in Remark \ref{rmk32}, we may take $\lambda=1$ in \eqref{303} and $\lambda=\frac{1}{2}$ in \eqref{304}. Direct computation then shows that $a<\frac{1}{8}=\frac{5}{\lambda_5}$. So (\ref{ainduction}) holds for $n=3$.

By induction, we now suppose that $a\leq \frac{5}{\lambda_n}$ for some $n\geq 3$.  To prove that  $a_n \leq \frac{5}{\lambda_{n+1}}$, we assume the contrary, $ a_n > \frac{5}{\lambda_{n+1}}$.    We will  derive a contradiction, which proves  $a\le\frac{5}{\lambda_{n+1}}$.

By Theorem \ref{A}, we have
\begin{align*}
&\sum_{k=2}^{n}(\lambda_{n+1}-\lambda_k+\frac{4}{15\alpha})(\lambda_k+2)b_{k}^{2}\\
\leq& \sum_{\substack{k=2\\ k\text{ odd}}}^{n}(\lambda_{n+1}-\lambda_k+\frac{4}{15\alpha})\frac{2k+3}{2\alpha^2}(a-0.089\lambda_k(2\lambda^2-2\lambda+1)a^2)^2\\
+&\sum_{\substack{k=2\\ k\text{ even}}}^{n}(\lambda_{n+1}-\lambda_k+\frac{4}{15\alpha})\frac{2k+3}{2\alpha^2}((0.081+0.919\lambda)a-0.089\lambda^2\lambda_k a^2)^2\\
=:&\frac{1}{2\alpha^2}f_n(\lambda)a^2.
\end{align*}

The following lemma implies that the worst case happens when $\lambda=1$.

\begin{lemma}
\label{Lemma4.1}
If  $\frac{5}{\lambda_{n+1}}\leq a\leq \frac{5}{\lambda_n}$ for some $n\geq 3$, then we have
\begin{equation*}
f_n(\lambda)\leq f_n(1),\text{ for any } \frac{1}{2}\leq \lambda \leq 1.
\end{equation*}
\end{lemma}
\begin{proof}
We first assume $n$ is odd and relabel $n$ to be $2n+1$, $n\geq 1$. By direct computation, we have
\begin{align*}
f_{2n+1}(\lambda)=
&\sum_{m=1}^{n}\left[(\lambda_{2n+2}-\lambda_{2m}+\frac{4}{15\alpha})(4m+3)((0.081+0.919\lambda)-0.089\lambda^2\lambda_{2m} a)^2\right.\\
&\left.+(\lambda_{2n+2}-\lambda_{2m+1}+\frac{4}{15\alpha})(4m+5)(1-0.089\lambda_{2m+1}(2\lambda^2-2\lambda+1)a)^2\right]\\
&=(0.081+0.919\lambda)^2 S_1+S_2-0.178\lambda^2(0.081+0.919\lambda)a S_3\\
&-0.178(2\lambda^2-2\lambda+1)a S_4+0.089^2\lambda^4 a^2 S_5+0.089^2(2\lambda^2-2\lambda+1)^2a^2S_6,&
\end{align*}
where $S_i,i=1,...,6$ are given by
\begin{equation*}
S_1=\sum_{m=1}^{n}\left(\lambda_{2n+2}-\lambda_{2m}+\frac{4}{15\alpha}\right)(4m+3)=4 n^4+28 n^3+\left(\frac{8}{15 \alpha }+59\right)n^2+\left(\frac{4}{3 \alpha }+35\right)n,
\end{equation*}

\begin{align*}
S_2
&=\sum_{m=1}^{n}\left(\lambda_{2n+2}-\lambda_{2m+1}+\frac{4}{15\alpha}\right)(4m+5)\\
&=4 n^4+28 n^3+\left(\frac{8}{15 \alpha }+51\right) n^2+\left(\frac{28}{15 \alpha }+7\right) n,
\end{align*}

\begin{align*}
S_3
&=\sum_{m=1}^{n}\left(\lambda_{2n+2}-\lambda_{2m}+\frac{4}{15\alpha}\right)(4m+3)\lambda_{2m}\\
&=
\frac{16 n^6}{3}+56 n^5+\left(\frac{16}{15 \alpha }+\frac{676}{3}\right)n^4+\left(\frac{16}{3 \alpha }+434 \right)n^3\\
&+\left(\frac{124}{15 \alpha }+\frac{1198}{3}\right)n^2+\left(\frac{4}{\alpha }+140\right)n,
\end{align*}

\begin{align*}
S_4
&=\sum_{m=1}^{n}(\lambda_{2n+2}-\lambda_{2m+1}+\frac{4}{15\alpha})(4m+5)\lambda_{2m+1}\\
&=
\frac{16 n^6}{3}+56 n^5+\left(\frac{16}{15 \alpha }+\frac{724}{3}\right)n^4+\left(\frac{112}{15 \alpha }+546 \right)n^3\\
&+\left(\frac{268}{15 \alpha }+\frac{1810}{3}\right)n^2+\left(\frac{84}{5 \alpha }+168\right)n,
\end{align*}

\begin{align*}
S_5
&=\sum_{m=1}^{n}\left(\lambda_{2n+2}-\lambda_{2m}+\frac{4}{15\alpha}\right)(4m+3)\lambda_{2m}^2\\
&=
\frac{32 n^8}{3}+\frac{448 n^7}{3}+\left(\frac{128}{45 \alpha }+848\right)n^6+\left(\frac{64}{3 \alpha }+\frac{7504 }{3}\right)n^5+\left(\frac{2624}{45 \alpha }+4062\right)n^4\\
&+\left(\frac{208}{3 \alpha }+\frac{10528}{3}\right)n^3+\left(\frac{1448}{45 \alpha }+\frac{4138}{3}\right)n^2+\left(\frac{8}{3 \alpha }+140\right)n,
\end{align*}

\begin{align*}
S_6
&=\sum_{m=1}^{n}\left(\lambda_{2n+2}-\lambda_{2m+1}+\frac{4}{15\alpha}\right)(4m+5)\lambda_{2m+1}^2\\
&=
\frac{32 n^8}{3}+\frac{448 n^7}{3}+\left(\frac{128}{45 \alpha }+912\right)n^6+\left(\frac{448}{15 \alpha }+\frac{9520}{3}\right)n^5+\left(\frac{5504}{45 \alpha }+6830\right)n^4\\
&+\left(\frac{1232}{5 \alpha }+\frac{27496}{3}\right)n^3+\left(\frac{11384}{45 \alpha }+\frac{20962 }{3}\right)n^2+\left(\frac{616}{5 \alpha }+1932\right)n.
\end{align*}

Direct computations yield
\begin{align*}
f_{2n+1}'(\lambda)
&\geq 1.838(0.081+0.919\lambda)S_1+S_2-0.178(0.162\lambda+2.757\lambda^2)\frac{5}{\lambda_{2n+1}} S_3\\
&-0.178(4\lambda-2)\frac{5}{\lambda_{2n+1}} S_4+0.178^2\lambda^3 \frac{25}{\lambda_{2n+2}^{2}} S_5\\
&+0.356^2(4\lambda^3-6\lambda^2+4\lambda+1)\frac{25}{\lambda_{2n+2}^{2}}S_6\\
&> 0.
\end{align*}

Hence $f_{2n+1}(\lambda)\leq f_{2n+1}(1)$ and Lemma \ref{Lemma4.1}  is thus proved when $n$ is odd. The case when $n$ is even follows from a similar computation and we omit the details.
\end{proof}

By Lemma \ref{Lemma4.1} above, we infer that
\begin{align*}
&\sum_{k=2}^{n}(\lambda_{n+1}-\lambda_k+\frac{4}{15\alpha})(\lambda_k+2)b_{k}^{2}\\
\leq& \sum_{k=2}^{n}(\lambda_{n+1}-\lambda_k+\frac{4}{15\alpha})\frac{2k+3}{2\alpha^2}(a-0.089\lambda_k a^2)^2\\
\leq& \sum_{k=2}^{n}(\lambda_{n+1}-\lambda_k+\frac{4}{15\alpha})\frac{2k+3}{2\alpha^2}(1-\frac{0.445\lambda_k}{\lambda_{n+1}})^2a^2.
\end{align*}

Expanding the above summation, we get
\begin{align*}
& \sum_{k=2}^{n}(\lambda_{n+1}-\lambda_k+\frac{4}{15\alpha})\frac{2k+3}{2\alpha^2}(1-\frac{0.445\lambda_k}{\lambda_{n+1}})^2\\
&\leq 0.37n^4+0.37n^3+(\frac{0.166}{\alpha}+5.4)n^2+(\frac{0.75}{\alpha}-19.12)n-(\frac{1.35}{\alpha}-17.8)\\
&-(\frac{5.75}{\alpha}-101)\frac{1}{\lambda_{n+1}}-(\frac{0.46}{\alpha}-7.05)\frac{2n^2+10n+17}{\lambda_{n+1}^2}\\
&\leq 0.37n^4+3.7n^3+5.74n^2-17.62n+15.2+\frac{101}{\lambda_{n+1}}+\frac{56}{\lambda_{n+1}^2}.
\end{align*}

Plugging into \eqref{405}, we obtain

\begin{align*}
&0\leq \left[\frac{5}{6\alpha^2}(3\lambda_{n+1}-12+\frac{32}{5\alpha})+\frac{28}{15\alpha^2}(1-\frac{89}{20\lambda_{n+1}})^2\right.\\
&\left.+\frac{1}{2\alpha^2}\left(0.37n^4+3.7n^3+5.74n^2-17.62n+15.2+\frac{101.4}{\lambda_{n+1}}+\frac{56}{\lambda_{n+1}^2}\right)\right]a^2\\
&-\left[\frac{2}{3\alpha}\left(-\frac{8}{3\alpha^2}+\frac{136}{3\alpha}-60\right)+\frac{2}{3\alpha^2}\left(3\lambda_{n+1}-12+\frac{32}{5\alpha}\right)\right]a+\frac{8}{15\alpha}\left(-\frac{8}{3\alpha^2}+\frac{136}{3\alpha}-60\right)\\
&=:h_n(a),
\end{align*}
where $h_n (a)$ is defined at the last equality. Note that $h_n (a)$ is quadratic in $a$ and convex.  To get a contradiction, it suffices to check that the parabola $h_n(a)$ is negative at both $a=\frac{5}{\lambda_{n+1}}$ and $a=\frac{5}{\lambda_n}$.

When $a=\frac{5}{\lambda_{n+1}}$, we have
\begin{align*}
h_n(\frac{5}{\lambda_{n+1}})
&=\left(\frac{37n^4}{8\alpha^2\lambda_{n+1}^2}-\frac{10}{\alpha^2}+\frac{8}{15\alpha}\left(-\frac{8}{3\alpha^2}+\frac{136}{3\alpha}-60\right)\right)+\frac{185n^3}{4\alpha^2\lambda_{n+1}^2}\\
&+\left[\frac{287n^2}{4\alpha^2 \lambda_{n+1}^2}+\left(\frac{40}{\alpha}+\frac{725}{18\alpha^2}-\frac{112}{45\alpha^3}\right)\frac{1}{\lambda_{n+1}}\right]-\frac{881n}{4\alpha^2\lambda_{n+1}^{2}}\\
&+\left[\frac{140}{3\alpha^2}(1-\frac{89}{20\lambda_{n+1}})^2+\left(-\frac{60}{\alpha^2}+\frac{400}{3\alpha^3}\right)\right]\frac{1}{\lambda_{n+1}^{2}}+\frac{2535}{2\alpha^2\lambda_{n+1}^{3}}+\frac{700}{\alpha^2\lambda_{n+1}^{4}}
\end{align*}.

Note that the leading order term is negative (in fact it is less than $-\frac{1}{6}$) and direct computation yields that $h_n(\frac{5}{\lambda_{n+1}})<0$. By similar computations we also derive that $h_n(\frac{5}{\lambda_{n}})<0$.

As a consequence, we get $a\leq \frac{5}{\lambda_n}$ for any $n\geq 3$. Let $n$ tend to infinity, we get $a=0$, which implies that $\frac{1}{\alpha}-\beta=0$, a contradiction.

\appendix
\section{proof of Lemma \ref{lem32}}\label{appendix}
In this appendix, we prove Lemma~\ref{lem32}.

\begin{proof}[Proof of Lemma~\ref{lem32}]
Define $A_{m,n}^+=\int_0^1 x^m(1-x^2)^ng$, $A_{m,n}^-=\int_{-1}^0 |x|^m(1-x^2)^ng$, and $A_{m,n}=A_{m,n}^++A_{m,n}^-$. We begin with the estimate of $A_3$. By definition,
\begin{eqnarray*}
    A_3=\int_{-1}^{1}(1-x^2)g\widetilde{F}_{3}'=\frac{1}{6}\int_{-1}^{1}(1-x^2)(7x^2-1)g=\frac{1}{6}(7A_{2,1}-a).
\end{eqnarray*}
By Cauchy-Schwartz inequality and \eqref{308},
\begin{align*}
    (A_{2,1}^+)^2&\le \int_0^1(1-x^2)^2g \int_0^1 x^4g\\
    &\le (a_+-A_{2,1}^+)(\frac{a+1}{2}-a_+-A_{2,1}^+),
\end{align*}
so \begin{equation}\label{A01}
    A_{2,1}^+\le a_+-\frac{2a_+^2}{a+1}.
\end{equation}
In the same way,
\begin{equation*}
    A_{2,1}^-\le a_--\frac{2a_-^2}{a+1}.
\end{equation*}
Hence,
\begin{equation*}
    A_{2,1}\le a-\frac{2a_+^2+2a_-^2}{a+1}=a-\frac{2a^2}{a+1}(2\lambda^2-2\lambda+1).
\end{equation*}
Therefore
\begin{eqnarray*}
    A_3\le a-\frac{7}{3}\frac{a^2}{a+1}(2\lambda^2-2\lambda+1),
\end{eqnarray*}
which, together with the definition of $A_3$, implies
\begin{eqnarray*}
    |A_3|\le \max\{ a-\frac{7}{3}\frac{a^2}{a+1}(2\lambda^2-2\lambda+1), \frac{a}{6}\}.
\end{eqnarray*}
For $A_2$, we have
\begin{eqnarray*}
    |A_2|=|\int_{-1}^{1}x(1-x^2)g|\leq \max\left\{A_{1,1}^+,A_{1,1}^-\right\}.
\end{eqnarray*}
By Cauchy-Schwartz inequality,
\begin{align*}
(A_{1,1}^+)^2
\leq A_{2,1}^+\int_{0}^{1}(1-x^2)g\le a_+^2-\frac{2a_+^3}{a+1}.
\end{align*}
\begin{align*}
(A_{1,1}^-)^2
\leq A_{2,1}^-\int_{0}^{1}(1-x^2)g\le a_-^2-\frac{2a_-^3}{a+1}.
\end{align*}
Since we have assumed $\lambda\ge \frac{1}{2}$, we conclude that
\begin{eqnarray*}
    |A_2|\le a_+\sqrt{1-\frac{2a_+}{a+1}}.
\end{eqnarray*}
The estimate of $|A_4|$ is similar to that of $|A_2|$. By definition,
\begin{eqnarray*}
    A_4=\int_{-1}^{1}(1-x^2)g\widetilde{F}_{4}'=\frac{1}{2}\int_{-1}^{1}(1-x^2)(3x^2-1)xg=A_{1,1}-\frac{3}{2}A_{3,1}.
\end{eqnarray*}
By Cauchy-Schwartz inequality and \eqref{309},
\begin{align*}
    A_{3,1}\ge \frac{(A_{1,1}^+)^2}{\int_0^1 xg}\ge 2(A_{1,1}^+)^2,
\end{align*}
so
\begin{equation*}
    A_4^+\le A_{1,1}^+-3(A_{1,1}^+)^2
\end{equation*}
On the other hand,
\begin{align*}
    A_4^+\ge \frac{1}{2}\min_{0\le x\le1}\{(3x^2-1)x\}\int_{0}^{1}(1-x^2)g=-\frac{1}{9}a_+.
\end{align*}
In the same way,
\begin{eqnarray*}
   -(A_{1,1}^--3(A_{1,1}^-)^2) \le A_4^-\le \frac{1}{9}a_-
\end{eqnarray*}
Since $\lambda\ge \frac{1}{2}$, we conclude that
\begin{eqnarray*}
    |A_4|\le  A_{1,1}^+-3(A_{1,1}^+)^2+\frac{1}{9}a_-.
\end{eqnarray*}
Finally, for $A_5$, we have
\begin{align*}
    A_5=\frac{1}{16}\int_{-1}^1(1-x^2)(1-18x^2+33x^4)g=\frac{1}{16}(16a-33A_{2,2}-15A_{2,0})
\end{align*}

By Cauchy-Schwartz inequality and \eqref{308},
\begin{align*}
    A_{2,2}^+\ge \frac{(A_{2,1}^+)^2}{\int_0^1 x^2g}\ge \frac{ (A_{2,1}^+)^2}{\frac{a+1}{2}-a_+},
\end{align*}
so by \eqref{A01},
\begin{align*}
    A_5^+&\le \frac{1}{16}\big(16a_+-\frac{ 33(A_{2,1}^+)^2}{\frac{a+1}{2}-a_+}-15(a_+-A_{2,1}^+\big)\\
    &\le \frac{1}{16}\Big(a_+-3(a_+-\frac{2a_+^2}{a+1})(\frac{22a_+}{a+1}-5)\Big)\\
    &=a_+-\frac{6a_+^2}{a+1}+\frac{33a_+^3}{4(a+1)^2}.
\end{align*}
Therefore
\begin{align*}
    A_5\le a-\frac{6(a_+^2+a_-^2)}{a+1}+\frac{33(a_+^3+a_-^3)}{4(a+1)^2}.
\end{align*}
On the other hand,
\begin{align*}
    A_5\ge \frac{1}{16}\min_{-1\le x\le1}\{1-18x^2+33x^4\}\int_{-1}^{1}(1-x^2)g=-\frac{1}{11}a.
\end{align*}
From \eqref{406} and the estimates of $|A_2|$ and $|A_3|$, we can deduce that  $a<0.125$, so now it is not hard to see that
\begin{align*}
    |A_5|\le a-\frac{6(a_+^2+a_-^2)}{a+1}+\frac{33(a_+^3+a_-^3)}{4(a+1)^2}.
\end{align*}
Thus the proof of Lemma \ref{lem32} is completed.
\end{proof}

\section{proof of Lemma \ref{lem34} and Lemma \ref{lem35}}\label{appendixB}
In this appendix we prove Lemma \ref{lem34} and Lemma \ref{lem35}. The proofs are technical and use many quantitative properties of Gegenbauer polynomials.

 Before we prove Lemma~\ref{lem34}, we first state some general lemma about Gegenbauer polynomials.  Denote by $x_{nk}(\nu)$, $k=1,\cdots ,n$, the zeros of $C_n^\nu(x)$ enumerated in decreasing order, that is, $1>x_{n1}(\nu)>\cdots>x_{nn}(\nu)>-1$.

\begin{lemma}[Corollary 2.3 in Area et al.\cite{Area2004}]\label{lemA1}
    For any $n \ge 2$ and for every $\nu\ge 1$, the inequality
\begin{align}
    x_{n1}(\nu)\le \sqrt{\frac{(n-1)(n+2\nu-2)}{(n+\nu-2)(n+\nu-1)}}\cos(\frac{\pi}{n+1})
\end{align}
    holds.
\end{lemma}
The next lemma is well-known and it is valid for many other orthogonal polynomials.
\begin{lemma}[Olver et al. \cite{Olver2010}]\label{lemA2}
Denote by $y_{nk}(\nu)$, $k=0,1,\cdots ,n-1,n$, the local maxima of $|C_n^\nu(x)|$ enumerated in decreasing order, then $y_{n0}(\nu)=1, y_{nn}(\nu)=-1$, and we have
\begin{enumerate}
\item[$(a)$]\quad
   $y_{nk}(\nu)=x_{n-1,k}(\nu+1),\ k=1,\cdots, n-1.$
\item [$(b)$]\quad
$|C_n^\nu(y_{n0}(\nu))|>|C_n^\nu(y_{n1}(\nu))|>\cdots>|C_n^\nu(y_{n,[\frac{n+1}{2}]}(\nu))|.$
\item[$(c)$]\quad
$(C_n^\nu)^{(k)}(x)>0$ on $(x_{n1}(\nu),1)$ for all $k=0,1,\cdots, n.$
\end{enumerate}
\end{lemma}

\begin{proof}[Proof of Lemma~\ref{lem34}]
Direct computation by Matlab shows that Lemma~\ref{lem34} holds for $6\le k\le 50$, so in what follows we may assume $k>50$.

By Lemma~\ref{lemA1} and \eqref{201}, we know that the minimum of $\widetilde{F}_k'$ on $[0,1]$ is achieved at the point
\begin{align}\label{A03}
    x_{k-2,1}(\frac{7}{2})\le \sqrt{\frac{(k-3)(k+3)}{(k-\frac{1}{2})(k+\frac{1}{2})}}\cos(\frac{\pi}{k-1})\le 1-\frac{9.1}{k^2}.
\end{align}

 Taking $N=2$ in Lemma~\ref{lem33}, we obtain
    \begin{align}
     \widetilde{F}_k'(\cos{\zeta})
     &=\frac{24}{k(k+1)(k+2)(k+3)}C_{k-1}^{\frac{5}{2}}(\cos{\zeta})\notag\\
     &=8\sqrt{\frac{2}{\pi}}\Big((\sin{\zeta})^{-\frac{5}{2}}\Gamma(k)\Big(\frac{\cos{(\delta_{k-1,0})}}{\Gamma(k+\frac{5}{2})}+\frac{15}{8}\frac{\cos{(\delta_{k-1,1})}}{\Gamma(k+\frac{7}{2})\sin{\zeta}}\Big)+\widetilde{R}\Big)\notag\\
     &=8\sqrt{\frac{2}{\pi}}\Big(\frac{k^{\frac{5}{2}}\Gamma(k)}{{l^{\frac{5}{2}}}\Gamma(k+\frac{5}{2})}\big(\cos{((k+\frac{3}{2})\zeta-\frac{5\pi}{4})}
     +\frac{15}{8l}\frac{k}{(k+\frac{5}{2})}\cos{((k+\frac{5}{2})\zeta-\frac{3\pi}{4})}\big)
     +\widetilde{R}\Big)\label{A04},
\end{align}
where $\widetilde{R}$ satisfies
\begin{align}\label{A05}
    |\widetilde{R}|\le \frac{15}{8} \frac{\Gamma(k)}{\Gamma(k+\frac{9}{2})}(\sin{\zeta})^{-\frac{9}{2}}\cdot
\begin{cases}
\sec{\zeta}&\text{ if }0<\zeta\leq\frac{\pi}{4},\\
2\sin{\zeta}&\text{ if }\frac{\pi}{4}<\zeta<\frac{\pi}{2}.
\end{cases}
\end{align}
Let $\sin{\zeta}=\frac{l}{k}$. Then by \eqref{A03} we can assume $l\ge \sqrt{18}$. From \eqref{A05} we know that if $l\le \frac{k}{\sqrt{2}}$, then
\begin{align}\label{A06}
    |\widetilde{R}|\le \frac{15}{8l^\frac{9}{2}} \frac{k^\frac{9}{2}\Gamma(k)}{\Gamma(k+\frac{9}{2})}\frac{1}{\sqrt{1-\frac{l^2}{k^2}}}
    < \frac{15}{8l^\frac{9}{2}} \frac{1}{\sqrt{1-\frac{l^2}{k^2}}};
\end{align}
while if $l> \frac{k}{\sqrt{2}}$, then
\begin{align}\label{A07}
    |\widetilde{R}|\le \frac{15}{4l^\frac{7}{2}} \frac{k^\frac{7}{2}\Gamma(k)}{\Gamma(k+\frac{9}{2})}
    < \frac{15(\sqrt{2})^\frac{7}{2}}{4k^\frac{9}{2}}.
\end{align}
To get the desired lower bound, we shall use the following simple estimates.
\begin{equation}\label{A08}
    \cos(x+\delta)=\cos x-\delta \sin(x+h\delta)\ge \cos x-|\delta|.
\end{equation}
\begin{equation}\label{A09}
    \zeta-\sin\zeta\le (\frac{\pi}{2}-1) \sin^3{\zeta} \le \sin^3{\zeta}, \ 0<\zeta<\frac{\pi}{2}.
\end{equation}
 With the help of \eqref{A08} and \eqref{A09}, we have
 \begin{align}
    \cos{((k+\frac{3}{2})\zeta-\frac{5\pi}{4})}
    &=\cos((k+\frac{3}{2})\frac{l}{k}+(k+\frac{3}{2})(\zeta-\sin\zeta)-\frac{5\pi}{4})\notag\\
    &\ge \cos(l-\frac{5\pi}{4})-((k+\frac{3}{2})(\zeta-\sin\zeta)+\frac{3l}{2k})\notag\\
    &\ge  \cos(l-\frac{5\pi}{4})-((k+\frac{3}{2})(\frac{l}{k})^3+\frac{3l}{2k}).\label{A10}
\end{align}
Similarly, we get
\begin{align}
    \cos{((k+\frac{5}{2})\zeta-\frac{3\pi}{4})}\ge \cos(l-\frac{3\pi}{4})-((k+\frac{5}{2})(\frac{l}{k})^3+\frac{5l}{2k}).\label{A11}
\end{align}
Therefore it holds that
\begin{align*}
   &l^{-\frac{5}{2}}\Big(\cos{((k+\frac{3}{2})\zeta-\frac{5\pi}{4})}
     +\frac{15}{8l}\frac{k}{(k+\frac{5}{2})}\cos{((k+\frac{5}{2})\zeta-\frac{3\pi}{4})}\Big)\\
     \ge& -\big(\frac{15 \cos \left(l+\frac{\pi }{4}\right)}{8 l^{7/2}}+\frac{\cos \left(l-\frac{\pi }{4}\right)}{l^{5/2}}\big)-\Big(\frac{3}{2 l^{3/2} k}+\frac{75}{16 l^{5/2} k}+\frac{3 \sqrt{l}}{2 k^3}+\frac{75}{16 \sqrt{l} k^3}+\frac{\sqrt{l}}{k^2}+\frac{15}{8 \sqrt{l} k^2}\Big)\\
     =:&\varphi_k(l).
\end{align*}
Since $\frac{k^{\frac{5}{2}}\Gamma(k)}{\Gamma(k+\frac{5}{2})}<1$, we deduce that
\begin{equation*}
\widetilde{F}_k'(\cos{\zeta})\ge 8\sqrt{\frac{2}{\pi}}(\varphi_k^-(l)-|\widetilde{R}|),
\end{equation*}
where $\varphi_k^-(l)=\min\{\varphi_k (l),0\}$. Now we split $l$ into different intervals to get a relatively precise estimate.

If $\sqrt{18}\le l\le 5$, then it is straightforward to check that $\varphi_k(l)\ge 0$ and  so
\begin{equation*}
    \widetilde{F}_k'(\cos{\zeta})\ge -8\sqrt{\frac{2}{\pi}}(\frac{15}{8l^\frac{9}{2}} \frac{1}{\sqrt{1-\frac{l^2}{k^2}}})>-8\sqrt{\frac{2}{\pi}}\times 0.003>-0.081.
\end{equation*}

If $5<l\le 5.6$, then $\varphi_k^-(l)\ge -0.0059-0.004=-0.0099$, and \eqref{A06} implies $|\widetilde{R}|<0.0015$ and  so
\begin{equation*}
    \widetilde{F}_k'(\cos{\zeta})\ge -8\sqrt{\frac{2}{\pi}}\times 0.0114 >-0.081.
\end{equation*}

If $l\ge 6.8$, then  $\varphi_k^-(l)\ge -0.0087-0.003=-0.0117$, and either \eqref{A06} or \eqref{A07} implies $|\widetilde{R}|<0.0005$ and so
\begin{equation*}
    \widetilde{F}_k'(\cos{\zeta})\ge -8\sqrt{\frac{2}{\pi}}\times 0.0122>-0.081.
\end{equation*}

Finally, if $5.6<l<6.8$, then \eqref{A06} implies $|\widetilde{R}|<0.0009$. In this case, $0.532\pi<l-\frac{5\pi}{4}<0.915\pi$, and
$1.032\pi<l-\frac{3\pi}{4}<1.415\pi$. Now we can refine the estimate in \eqref{A08} to get better estimates than \eqref{A10} and \eqref{A11}. More precisely, we have
\begin{align}
    \cos{((k+\frac{3}{2})\zeta-\frac{5\pi}{4})}
    &=\cos((k+\frac{3}{2})\frac{l}{k}+(k+\frac{3}{2})(\zeta-\sin\zeta)-\frac{5\pi}{4})\notag\\
    &\ge \cos(l-\frac{5\pi}{4})-\sin(l-\frac{5\pi}{4})((k+\frac{3}{2})(\zeta-\sin\zeta)+\frac{3l}{2k})\notag\\
    &\ge  \cos(l-\frac{5\pi}{4})-\frac{2l}{k}\sin(l-\frac{5\pi}{4}).\label{A12}
\end{align}
Similarly there holds
\begin{align}\label{A13}
    \cos{((k+\frac{5}{2})\zeta-\frac{3\pi}{4})}\ge \cos(l-\frac{3\pi}{4})-\frac{5l}{2k}\sin(l-\frac{3\pi}{4}).
\end{align}
Now it is straightforward to compute
\begin{align*}
    &l^{-\frac{5}{2}}\Big(\cos{((k+\frac{3}{2})\zeta-\frac{5\pi}{4})}
     +\frac{15}{8l}\frac{k}{(k+\frac{5}{2})}\cos{((k+\frac{5}{2})\zeta-\frac{3\pi}{4})}\Big)\\
     \ge&  l^{-\frac{5}{2}}\Big(\cos(l-\frac{5\pi}{4})-\frac{2l}{k}\sin(l-\frac{5\pi}{4})\Big)
     +\frac{15}{8}\frac{l^{-\frac{7}{2}}k}{(k+\frac{5}{2})}\Big(\cos(l-\frac{3\pi}{4})-\frac{5l}{2k}\sin(l-\frac{3\pi}{4})\Big)\\
     \ge& -0.01,
\end{align*}
so
\begin{equation*}
    \widetilde{F}_k'(\cos{\zeta})\ge -8\sqrt{\frac{2}{\pi}}\times (0.01+0.0009)>-0.081,
\end{equation*}
which completes the proofs of Lemma \ref{lem34}.
\end{proof}

\begin{proof}[Proof of Lemma~\ref{lem35}]
    The proof is similar to that of Lemma~\ref{lem34} above. We first prove the following estimate at one point:
    \begin{equation}\label{A14}
        0.081\le\widetilde{F}_k'(1-\frac{10}{\lambda_k})\le 0.11, \quad k\ge 6.
    \end{equation}
    Direct computation by Matlab shows that Lemma~\ref{lem34} holds for $6\le k\le 50$, so in what follows we may assume $k> 50$. The main tool we use is still \eqref{A04}, and the only difference is that now $\cos\zeta=1-\frac{10}{\lambda_k}$.

By Taylor expansion, one easily obtains
\begin{eqnarray}\label{A15}
    \frac{2\sqrt{5}}{k}-\frac{3\sqrt{5}}{k^2}<\zeta<\frac{2\sqrt{5}}{k},
\end{eqnarray}
and
\begin{align}\label{A16}
    \sin\zeta\ge \frac{2 \sqrt{5}}{k}-\frac{3 \sqrt{5}}{k^2}.
\end{align}
By \eqref{A05} and \eqref{A16}, there holds
\begin{align}\label{A17}
    |\widetilde{R}|&\le \frac{105}{128} \frac{\Gamma(k)}{\Gamma(k+\frac{9}{2})}(\frac{2 \sqrt{5}}{k}-\frac{3 \sqrt{5}}{k^2})^{-\frac{9}{2}} (1-\frac{10}{\lambda_k})^{-1}\notag\\
    &\le \frac{105}{128}\times 20^{-\frac{9}{4}} (1 -\frac{1.5}{k})^{-\frac{9}{2}} (1-\frac{10}{\lambda_k})^{-1}\notag\\
    &\le 0.001.
\end{align}
   We observe that $0.17\pi<(k+\frac{3}{2})\zeta-\frac{5\pi}{4}<0.2\pi$, and $0.67\pi<(k+\frac{3}{2})\zeta-\frac{3\pi}{4}<0.7\pi$. Then, similar to the estimates in \eqref{A12} and \eqref{A13}, we have

\begin{align}\label{A18}
    \cos{((k+\frac{3}{2})\zeta-\frac{5\pi}{4})}&\le \cos{((k+\frac{3}{2})(\frac{2\sqrt{5}}{k}-\frac{3\sqrt{5}}{k^2})-\frac{5\pi}{4})}\notag\\
    &\le\cos(2\sqrt{5}-\frac{5\pi}{4})-\frac{9\sqrt{5}}{2k^2}\sin(2\sqrt{5}-\frac{5\pi}{4}),
\end{align}
\begin{align}\label{A19}
    \cos{((k+\frac{3}{2})\zeta-\frac{5\pi}{4})}&\ge \cos{(\frac{2\sqrt{5}}{k}(k+\frac{3}{2})-\frac{5\pi}{4})}\notag\\
    &\ge\cos(2\sqrt{5}-\frac{5\pi}{4})-\frac{3\sqrt{5}}{k}\sin(2\sqrt{5}-\frac{5\pi}{4}+\frac{3\sqrt{5}}{k}),
\end{align}
\begin{align}\label{A20}
    \cos{((k+\frac{5}{2})\zeta-\frac{3\pi}{4})}&\le\cos{((k+\frac{5}{2})(\frac{2\sqrt{5}}{k}-\frac{3\sqrt{5}}{k^2})-\frac{3\pi}{4}))}\notag\\
    &\le\cos(2\sqrt{5}-\frac{3\pi}{4})
    -(\frac{2\sqrt{5}}{k}-\frac{15\sqrt{5}}{2k^2})\sin(2\sqrt{5}-\frac{3\pi}{4}+\frac{2\sqrt{5}}{k}),
\end{align}
\begin{align}\label{A21}
    \cos{((k+\frac{3}{2})\zeta-\frac{3\pi}{4})}&\ge\cos{((\frac{2\sqrt{5}}{k}(k+\frac{5}{2})-\frac{3\pi}{4})}\notag \\
    &\ge\cos(2\sqrt{5}-\frac{3\pi}{4})-
    \frac{5\sqrt{5}}{k}\sin(2\sqrt{5}-\frac{3\pi}{4}).
\end{align}

So by \eqref{A18} and \eqref{A20}, we obtain
\begin{align*}
    &(\sin{\zeta})^{-\frac{5}{2}} \frac{\Gamma(k)}{\Gamma(k+\frac{5}{2})}
    \Big(\cos{(\delta_{k-1,0})}+\frac{15}{8}\frac{\cos{(\delta_{k-1,1})}}{(k+\frac{5}{2})\sin{\zeta}}\Big)\\
    \le&(2\sqrt{5}-\frac{3 \sqrt{5}}{k})^{-\frac{5}{2}}\frac{k^{\frac{5}{2}}\Gamma(k)}{\Gamma(k+\frac{5}{2})}
    \Big(0.8551-\frac{5.218}{k^2}
    +\frac{15}{8}\frac{-0.5186-0.81(\frac{2\sqrt{5}}{k}-\frac{15\sqrt{5}}{2k^2})}{(k+\frac{5}{2})(\frac{2\sqrt{5}}{k}-\frac{3 \sqrt{5}}{k^2})}  \Big)\\
    \le& 20^{-\frac{5}{4}}
    (0.638-\frac{2}{k})\\
    \le& 0.0151,
\end{align*}
which, together with \eqref{A17}, implies that
\begin{align*}
   \widetilde{F}_k'(1-\frac{10}{\lambda_k})
   \le 8\sqrt{\frac{2}{\pi}}(0.0151+0.001)
    \le 0.103\le 0.11.
\end{align*}

The other direction $\widetilde{F}_k'(1-\frac{10}{\lambda_k})>0.081$ is similar. The only difference is that we need to use \eqref{A19} and \eqref{A21} instead of \eqref{A18} and \eqref{A20}. We omit the details. Thus \eqref{A14} is proved.

Now in view of Lemma~\ref{lem34}, we see that $\widetilde{F}_k'(1-\frac{10}{\lambda_k})>-\min\limits_{0\le x\le 1} \widetilde{F}_k'(x)$. Then by Lemma~\ref{lemA2} $(b)$,  $\widetilde{F}_k'(1-\frac{10}{\lambda_k})\ge \widetilde{F}_k'(x)$ for all $0\le x\le 1-\frac{10}{\lambda_k}$. Moreover, the convexity of $\widetilde{F}_k'(x)$ on $[1-\frac{10}{\lambda_k}, 1]$ is guaranteed by Lemma~\ref{lemA2} $(c)$. Thus the proof of Lemma \ref{lem35} is completed.
\end{proof}

\medskip

\section*{Acknowledgements}

J. Wei is partially supported by
NSERC of Canada. We thank Professor Changfeng Gui for interesting discussions.

\end{document}